\documentclass[12pt,reqno]{amsart}
\usepackage{amsmath, amssymb, amsfonts}
\numberwithin{equation}{section}
\newtheorem{thm}{Theorem}[section]
\newtheorem{prop}[thm]{Proposition}
\newtheorem{lem}[thm]{Lemma}

\newtheorem{cor}[thm]{Corollary}

\usepackage[a4paper, top = 1.2in, bottom = 1.2in, left = 1.2in, right = 1.2in]{geometry}

\numberwithin{equation}{section}

\newcommand{\A}{\mathbb{A}}
\newcommand{\C}{\mathbb{C}}
\newcommand{\N}{\mathbb{N}}
\newcommand{\Q}{\mathbb{Q}}
\newcommand{\R}{\mathbb{R}}

\newcommand{\Z}{\mathbb{Z}}

\newcommand{\f}{\mathbf{f}}
\newcommand{\g}{\mathbf{g}}

\newcommand{\mcO}{\mathcal{O}}

\newcommand{\mfa}{\mathfrak{a}}
\newcommand{\mfb}{\mathfrak{b}}
\newcommand{\mfc}{\mathfrak{c}}
\newcommand{\mfd}{\mathfrak{d}}

\newcommand{\h}{\mathfrak{h}}

\newcommand{\mfm}{\mathfrak{m}}
\newcommand{\mfn}{\mathfrak{n}}

\newcommand{\Dif}{\mathfrak{D}}
\newcommand{\mfp}{\mathfrak{p}}
\newcommand{\mfq}{\mathfrak{q}}

\newcommand{\GL}{\mathrm{GL}}
\newcommand{\SO}{\mathrm{SO}}

\def\1{1\!\!1}

\newcommand{\pmat}[4]{ \begin{pmatrix} #1 & #2 \\ #3 & #4 \end{pmatrix}}

\def\dis{\displaystyle}

\title[Simultaneous behaviour of the Fourier coefficients of HMCF]{Simultaneous behaviour of the Fourier coefficients of two Hilbert modular cusp forms}

\author[S. Kaushik]{Surjeet Kaushik}
\address[S. Kaushik]{Department of Mathematics, Indian Institute of Technology Hyderabad, Kandi, Sangareddy 502285, INDIA.}
\email{amarsurjeetkaushik@gmail.com}

\author[N. Kumar]{Narasimha Kumar}
\address[N. Kumar]{Department of Mathematics, Indian Institute of Technology Hyderabad, Kandi, Sangareddy 502285, INDIA.}
\email{narasimha.kumar@iith.ac.in}

\date{\today}
\begin{document}
\begin{abstract}
	In this article, we study the simultaneous sign changes  of the Fourier coefficients of two Hilbert cusp forms of different integral weights. 
	We also study the simultaneous non-vanishing of Fourier coefficients,  of two distinct non-zero primitive Hilbert cuspidal non-CM eigenforms of integral weights,
	at powers of a fixed prime ideal.
\end{abstract}

\maketitle

\section{Introduction}
The sign changes of Fourier coefficients of modular forms over number fields has been an interesting area of research in the recent years. 
In this article,  we are interested in the study of the simultaneous sign changes and simultaneous non-vanishing of the Fourier coefficients of 
distinct Hilbert cusp forms.

In~\cite{KS09}, the authors worked on simultaneous sign changes of Fourier coefficients of two cusp forms of different weights with real algebraic Fourier coefficients. 
They proved that, if $f$ and $g$ are two normalized cusp forms of the same level and different weights with totally real algebraic Fourier coefficients, 
then there exist a Galois automorphism $\sigma$ such that $f^\sigma$ and $g^\sigma$ have infinitely many Fourier coefficients of the opposite signs.
Their proof uses Landau's theorem on Dirichlet series with non-negative coefficients, the properties 
of the Rankin-Selberg zeta function attached to cusp forms, and the bounded denominators argument.

In~\cite{GKR15}, the authors, by using an elementary observation about real zeros of Dirichlet series instead of bounded denominators argument, strengthen the results of~\cite{KS09},
by doing away with the Galois conjugacy condition and in fact, they extended the result to cusp forms with arbitrary real Fourier coefficients. 

In~\cite{GKP18}, the authors investigated simultaneous non-vanishing of the Fourier coefficients at prime powers of Fourier coefficients of two different Hecke eigenforms of integral weight
over $\Q$.
They proved that if $f$ and $g$ are two Hecke eigenforms of integral weights and $a_f(n)$ and $a_g(n)$ are Fourier coefficients of $f$ and $g$, respectively, 
then for all primes $p$, the set $\{m\in \N| a_f(p^m)a_g(p^m)\neq0 \}$ has positive density. 

This article is a modest attempt to extend some results of Gun, Kohnen and Rath~\cite{GKR15} and Gun, Kumar and Paul~\cite{GKP18} to the Hilbert modular forms case.
Firstly, we show that two Hilbert cuspidal forms of different integral weights have infinitely many Fourier coefficients of same sign (resp., of opposite sign).
Secondly, we show that the simultaneous non-vanishing of the Fourier coefficients, of two non-zero distinct primitive Hilbert cuspidal non-CM eigenforms, at the powers of a fixed prime ideal
has a positive density.


\section{Preliminaries}	 
  We assume that $k=(k_1,\dots, k_n)\in \Z_{>0}^n$ throughout this section. For a non-archimedean place $\mfp$ of $F$, $F$ is a totally real field of degree $n$.
  Let $F_\mfp$ be a completion of $F$. Let $\mfa$ and $\mfb$ be integral ideals of $F$, and define a subgroup $K_\mfp(\mfa, \mfb)$ of $\GL_2(F_\mfp)$ as
  \[ K_\mfp(\mfa, \mfb)=\left\{\left(\begin{matrix} a & b \\ c & d \end{matrix}\right)\in \GL_2(F_\mfp)\, : \, \begin{matrix} a\in \mcO_\mfp, & b\in \mfa_\mfp^{-1}\Dif_\mfp^{-1}, & \\ c\in \mfb_\mfp\Dif_\mfp, & d\in\mcO_\mfp, & |ad-bc|_\mfp=1\end{matrix}\right\}\]
  where the subscript $\mfp$ means the $\mfp$-parts of given ideals. Furthermore, we put 
  \[ K_0(\mfa, \mfb)=\SO(2)^n\cdot\prod_{\mfp<\infty}K_\mfp(\mfa, \mfb) \quad \text{and} \quad W(\mfa, \mfb)=\GL_2^+(\R)^nK_0(\mfa, \mfb).\]
  In particular, if $\mfa=\mcO_F$, then we simply write $K_\mfp(\mfb):=K_\mfp(\mcO_F, \mfb)$, $W(\mfb):=W(\mcO_F, \mfb)$, etc.
  Then, we have the following disjoint decomposition of $\GL_2(\A_F)$:
  \begin{equation}\label{eqn:decomp}
  \GL_2(\A_F)=\cup_{\nu=1}^h\GL_2(F)x_\nu^{-\iota} W(\mfb),
  \end{equation}
  where $\dis x_\nu^{-\iota} =\left(\begin{matrix} t_\nu^{-1} & \\ & 1\end{matrix}\right)$ with $\{t_\nu\}_{\nu=1}^h$ taken to be a complete set of representatives of the narrow class group of $F$. We note that such $t_\nu$ can be chosen so that the infinity part $t_{\nu, \infty}$ is $1$ for all $\nu$. For each $\nu$, we also put
  \begin{align*}
  \Gamma_\nu(\mfb) &= \GL_2(F)\cap x_\nu W(\mfb)x_\nu^{-1} \\ 
  &= \left\{ \pmat{a}{t_\mu^{-1}b}{t_\nu c}{d}\in\GL_2(F): \, \begin{matrix} a\in \mcO_\mfp, & b\in \mfa_\mfp^{-1}\Dif_\mfp^{-1}, & \\ c\in \mfb_\mfp\Dif_\mfp, & d\in\mcO_\mfp, & |ad-bc|_\mfp=1\end{matrix}\right\}. 
  \end{align*} 
  
  Let $\psi$ be a Hecke character of $\A_F^\times$ whose conductor divides $\mfb$ and $\psi_\infty$ is of the form
  \[ \psi_\infty(x)={\rm sgn}(x_\infty)^k|x_\infty|^{i\mu},\]
  with $\mu\in\R^n$ and $\sum_{j=1}^n \mu_j=0$. We let $M_k(\Gamma_\nu(\mfb), \psi_\mfb, \mu)$ denote the space of all functions $f_\nu$ that are holomorphic on $\h^n$ and 
  at cusps, satisfying
  \[ f_\nu ||_k \gamma=\psi_\mfb(\gamma)\det \gamma^{i\mu/2}f_\nu \]
  for all $\gamma$ in $\Gamma_\nu(\mfb)$. We note that such a function $f_\nu$ has a Fourier expansion
  \[ f_\nu(z)=\sum_{\xi\in F}a_\nu(\xi) \exp(2\pi i \xi z)\]
  where $\xi$ runs over all the totally positive elements in $t_\nu^{-1}\mcO_F$ and $\xi=0$. 
  A Hilbert modular form is a cusp form, if for all $\gamma \in \GL^+_2(F)$, the constant term of $f||_k\gamma$
  in its Fourier expansion is $0$, and the space of cusp forms with respect to $\Gamma_{\nu}(\mfb)$ is denoted by $M_k(\Gamma_\nu(\mfb), \psi_\mfb, \mu)$.

  Now, put $\mathbf{f}:=(f_1,\dots,f_n)$ where $f_\nu$ belongs $M_k(\Gamma_\nu(\mfb), \psi_\mfb, \mu)$ for each $\nu$, and define $\mathbf{f}$ to be a function on $\GL_2(\A_F)$ as
  \[ \mathbf{f}(g)=\mathbf{f}(\gamma x_\nu^{-\iota}w):=\psi_\mfb(w^\iota)\det w_\infty^{i\mu/2}(f_\nu||_k w_\infty)(i\!\! i)\]
  where $\gamma x_\nu^{-\iota}w\in\GL_2(F)x_\nu^{-\iota}W(\mfb)$ as in (\ref{eqn:decomp}), and  $w^\iota:=\omega_0(^t w)\omega_0^{-1}$ with $\dis \omega_0=\left(\begin{matrix} & 1 \\ -1 & \end{matrix}\right)$. The space of such $\mathbf{f}$ is denoted as $M_k(\psi_\mfb, \mu)=\prod_\nu M_k(\Gamma_\nu(\mfb), \psi_\mfb, \mu)$. Furthermore, the space consisting of all  $\mathbf{f}=(f_1,\dots,f_n)\in M_k(\psi_\mfb, \mu)$ satisfying
  \[ \mathbf{f}(sg)=\psi(s)\mathbf{f}(g) \quad \text{for any}\, s\in \A_F^\times \quad \text{and}\quad x\in\GL_2(\A_F)\]
  is denoted as $M_k(\mfb,\psi)$.  If $f_\nu \in S_k(\Gamma_\nu(\mfb), \psi_\mfb, \mu)$ for each $\nu$, then the space of such $\mathbf{f}$ is denoted by $S_k(\mfb,\psi)$
  
  Let $\mfm$ be an integral ideal of $F$ and write $\mfm=\xi t_\nu^{-1}\mcO_F$ with a totally positive element $\xi$ in $F$. Then, we define the Fourier coefficients
  of $\f$ as
  \begin{equation}\label{eqn:coeff}
  C(\mfm,\mathbf{f}) :=\begin{cases} N(\mfm)^{k_0/2}a_\nu(\xi)\xi^{-(k+i\mu)/2} \quad & \text{if}\quad \mfm=\xi t_\nu^{-1}\mcO_F\subset \mcO_F \\
  0 & \text{if} \quad \mfm \, \text{is not integral} \end{cases} 
  \end{equation}
  where  $k_0=\max \{k_1,\dots,k_n\}$.  
  
  \section{Statements of the main results}
  In this section, we shall state the main results of this article. Firstly, we prove a result on the simultaneous sign changes of the Fourier coefficients 
  of two Hilbert cuspidal forms of different integral weights. More precisely, we prove:
  \begin{thm}
     \label{maintheorem1}
   Let $\f$ and $\g$ be non-zero Hilbert cusp forms over $F$ of level $\mfc$ and different integral weights $k=(k_1,\ldots,k_n)$, $l=(l_1,\ldots,l_n)$,
   respectively. For each integral ideal $\mfm \subseteq \mcO_F$, let $C(\mfm,\f)$ and $C(\mfm,\g)$ denote the Fourier coefficients (as defined in ~\eqref{eqn:coeff}) 
   of $\f$ and $\g$, respectively.    	
   Further, assume that $C(\mfm,\f)$, $C(\mfm,\g)$ are real numbers. If $C(\mcO_F,\f)C(\mcO_F,\g)\neq 0$, then there exist infinitely many ideals
   $\mfm \subseteq \mcO_F$ such that $C(\mfm,\f) C(\mfm,\g)>0$ and infinitely many ideals $\mfm \subseteq \mcO_F$ such that $C(\mfm,\f)C(\mfm,\g)<0$.
  \end{thm}
  The second result is about the simultaneous non-vanishing of the Fourier coefficients, of two non-zero distinct primitive Hilbert cuspidal eigenforms, at prime powers of a fixed prime ideal
  has a positive density. More precisely, we prove:
   \begin{thm}
  \label{maintheorem2}
	Let $\f$ and $\g$ be distinct primitive Hilbert cuspidal non-CM eigenforms over $F$ with trivial nebentypus and  
	of levels $\mfc_1,\mfc_2$ and with integral weights $k=(k_1,\ldots,k_n)$, $l=(l_1,\ldots,l_n)$,
	respectively. We further assume that $k_1\equiv \dots \equiv k_n \equiv l_1\equiv \dots \equiv l_n\equiv 0 \pmod{2}$ 
	and each $k_j, l_j \geq 2$.	
	
	For each ideal $\mfm \subseteq \mcO_F$, let $C(\mfm,\f)$ and $C(\mfm,\g)$ denote the Fourier coefficients (as defined in ~\eqref{eqn:coeff}) 
	of $\f$ and $\g$, respectively. Then, for any prime ideal $\mfp \subseteq \mcO_F$ such that $\mfp \nmid \mfc_1\mfc_2\Dif_F$, the set 
	$$\{m\in\N | C(\mfp^m,\f) C(\mfp^m,\g)\neq 0 \}$$
	has positive density. 
  \end{thm}
  \begin{cor}
   Assume that the hypothesis of the above theorem holds. Then, for any prime ideal $\mfp \nmid \mfc_1\mfc_2\Dif_F$, there exists infinitely $m \in \N$ such that $C(\mfp^m,\f) C(\mfp^m,\g)\neq 0$. 
  \end{cor}

\section{Proof of Theorem~\ref{maintheorem1}}
 For the proof of Theorem~\ref{maintheorem1}, we need the following basic results.
 \begin{lem}
 	\label{lem1}
 	Let $s\in\C$ and $$R(s)=\sum_{n\geq 1}\frac{a(n)}{n^s}$$ be a Dirichlet series with real coefficients $a(n)(n \in \N)$. 
 	Assume that $a(n)\geq0$ or $a(n)\leq0$ for all $n \geq 1$. 
 	If $R(s)$ has a real zero $\alpha$ in the region of convergence, then $R(s)$ is identical zero. 
 \end{lem}
 \begin{proof}
 	Without loss of generality, we can assume that $a(n)\geq0$ for all $n\geq 1$. 
 	Denote the sequence of partial sums of $R(\alpha)$ by $s_i=\sum_{n=1}^i\frac{a(n)}{n^\alpha}$, for $i \geq 1$. 
 	Since $a(n)\geq0$, the sequence $\{ s_i \}$ is a monotonically increasing sequence. 
 	Hence, the sequence $\{s_i\}$ converges to it's least upper bound. Since, $R(\alpha)=0$, we get that, for $i \geq 1$, $s_i$ is zero. 
 	We can deduce that each $a(i)=0$ for each $i$. Hence, $R(s)$ is identical zero. 
 	
 	Now, if $a(n)\leq0$ for all $n\geq1$, then we get the required result by applying above argument with $-R(s)$.
 \end{proof}
 
 \begin{lem}\rm(\cite[Lemma 6]{GKR15})
 \label{lem2}
   	Let $s\in\C$ and $a(n)\in\R$. For $m\geq1$, consider the Dirichlet polynomial $$R(s):=\sum_{1\leq n\leq m} \frac{a(n)}{n^s}.$$  
    	If $R(s)$ has infinitely many real zeros, then $R(s)$ is identically zero. 
 \end{lem}

 \begin{prop}\rm(\cite[Proposition 2.3]{Shi78})
 	For any integral ideal $\mfq\subseteq\mcO_F$ and every $\f\in S_k(\mfc,\psi)$, 
 	there exists an unique element of $S_k(\mfq\mfc,\psi)$, written as $\f|\mfq$, such that
 	\begin{equation}\label{Hecke:relation}
 	C(\mfm,\f|\mfq)=C(\mfq^{-1}\mfm,\f)
 	\end{equation}
 \end{prop}
 
 \begin{prop}\rm(\cite[Page 124]{Pan91})\label{Uoperator}
 	For any integral ideal $\mfq\subseteq\mcO_F$ and every $\f\in S_k(\mfc,\psi)$, 
 	there exists an unique element of $S_k(\mfq\mfc,\psi)$, written as $\f|U(\mfq)$, such that
 	\begin{equation}\label{U:operator}
 	C(\mfm,\f|U(\mfq))=C(\mfq\mfm,\f)
 	\end{equation} 
 \end{prop}
 We need the following proposition in the proof Theorem~\ref{maintheorem1}.
 \begin{prop}
 	\label{Key-Proposition}
 	Let $\f\in S_k(\mfc,\psi)$ and $\mfq$ be an integral ideal of $\mcO_F$. 
 	Then $\g=\f-(\f|U(\mfq))|\mfq$ is a Hilbert cusp form of weight $k$ and level $\mfq^2\mfc$.
 	Further, it has the property that $C(\mfm\mfq,\g)=0$ and $C(\mfm,\g)=C(\mfm,\f)$, if $(\mfm,\mfq)=1$. 
 \end{prop}
 
 \begin{proof}
 	Observe that $C(\mfm\mfq,\g)= C(\mfm\mfq, \f-(\f|U(\mfq))|\mfq) = C(\mfm\mfq,\f)-C(\mfm\mfq,(\f|U(\mfq))|\mfq)$.
 	Now, let us compute $C(\mfm\mfq,\f|U(\mfq)|\mfq) = C(\mfm, \f|U(\mfq)) = C(\mfm\mfq, \f)$. Hence, $C(\mfm\mfq,\g)=0$.
 	
 	Now, let us look at the expression when $(\mfm,\mfq)=1$. 
 	$$C(\mfm,\g)= C(\mfm, \f-(\f|U(\mfq))|\mfq) = C(\mfm,\f)-C(\mfm,(\f|U(\mfq))|\mfq).$$
        However, 
 	$C(\mfm,\f|U(\mfq)|\mfq)= C(\mfq^{-1}\mfm,\f|U(\mfq))= 0$, since $\mfq^{-1}\mfm$ is not an integral ideal.
 	Hence, $C(\mfm,\g)= C(\mfm, \f)$, if $(\mfm,\mfq)=1$.
 	
 \end{proof}

Now, we are in a position to prove Theorem~\ref{maintheorem1}. 
 \begin{proof}
 	By hypothesis, we have $C(\mcO_F,\f)C(\mcO_F,\g)\neq 0$.
 	First, we will show that there exist infinitely many $\mfm\subseteq\mcO_F$ such that 
 	\begin{equation}
 	\label{negative}
 	\frac{C(\mfm,\f)C(\mfm,\g)}{C(\mcO_F,\f)C(\mcO_F,\g)}<0
 	\end{equation}
 	Without loss of generality, we can assume that $C(\mcO_F,\f)C(\mcO_F,\g)>0$ 
 	as otherwise we replace $\g$ by $-\g$.
 	
 	If~\eqref{negative} is not true, then there exist an ideal $\mfm^\prime\subseteq\mcO_F$ such that 
 	\begin{equation}\label{positive}
 	C(\mfm,\f)C(\mfm,\g)\geq 0
 	\end{equation}
 	for all $\mfm\subseteq\mcO_F$ with $N(\mfm)\geq N(\mfm^\prime)$. 
 	Set $\mfn:=\prod_{N(\mfp)\leq N(\mfm^\prime)} \mfp$, where $\mfp$ are prime ideals of $\mcO_F$. 
 	
 	Suppose $\f_1$ and $\g_1$ are Hilbert modular cusp forms obtained from $\f$ and $\g$ respectively,
 	by applying the Proposition~\ref{Key-Proposition} to $\f$ and $\g$ with the ideal $\mfn$. 
 	Clearly, $\f_1$ and $\g_1$ are also Hilbert cusp forms of  level $k$ and $l$ respectively, and of level  $\mfc_1$. 
 	We just say that the level is $\mfc_1$, because as such we do not need the explicit level in the further calculations.
 	
 	For $s\in\C$ with $\mathrm{Re}(s)\gg 1$, the Rankin-Selberg $L$-function of $\f_1$ and $\g_1$ is defined by 
 	\begin{equation}\label{positive:cofficients}
 	R_{\f_1,\g_1}(s):=\sum_{\mfm\subseteq\mcO_F,(\mfm,\mfn)=1} \frac{C(\mfm,\f)C(\mfm,\g)}{N(\mfm)^s}.
 	\end{equation}
 	In above summation $C(\mfm,\f)C(\mfm,\g)\geq 0$, since, if $N(\mfm)\leq N(\mfm^\prime)$ then $\mfm=\prod_{\mfp_i|\mfn}\mfp_i^{e_i}$ implies  $(\mfm,\mfn)\neq1$. 
 	For $\mathrm{Re}(s)\gg 1$, we set 
 	$$L_{\f_1,\g_1}(s):=\zeta_F^{\mfc_1}(2s-(k_0+l_0)+2)R_{\f_1,\g_1}(s),$$
 	where $\zeta_F^{\mfc_1}(s)=\prod_{\mfp|\mfc_1, \mfp:\text{prime}}(1-N(\mfp)^{-s})\zeta_F(s)$, 
 	where $\zeta_F(s)=\sum_{\mfm\subseteq\mcO_F}N(\mfm)^{-s}$ is Dedekind zeta function of $F$. 
 	By the Euler expansion of Dedekind zeta function of $F$,  we get that
   \begin{align*}
   	\zeta_F^{\mfc_1}(s) =&\prod_{\mfp|\mfc_1, \mfp:\text{prime}}(1-N(\mfp)^{-s})\prod_{\mfp:\text{prime}}(1-N(\mfp)^{-s})^{-1}\\
   	                    &=\sum_{\mfm\subseteq\mcO_F,(\mfm,\mfc_1)=1}\frac{1}{N(\mfm)^{s}}=\sum_{n=1}^{\infty}\frac{a_n(\mfc_1)}{n^s},
   \end{align*}
   where $a_n(\mfc_1)$ is the number of integral ideals of norm $n$ that are co-prime to $\mfc_1$.
 	Hence, we can write
 	$$L_{\f_1,\g_1}(s)=\sum_{n=1}^{\infty}\frac{a_n(\mfc_1)n^{(k_0+l_0)-2}}{n^{2s}}\sum_{\mfm\subseteq\mcO_F,(\mfm,\mfn)=1} \frac{C(\mfm,\f)C(\mfm,\g)}{N(\mfm)^s}.$$
 	Now, we can re-write $$L_{\f_1,\g_1}(s)=\sum_{m=1}^\infty\frac{\mfb_m^{\mfc_1}(\f_1,\g_1)}{m^s},$$
 	where $$\mfb_m^{\mfc_1}(\f_1,\g_1)=\sum_{n^2|m}\left( a_n(\mfc_1)n^{(k_0+l_0)-2} \sum_{(\mfm,\mfn)=1,N(\mfm)=m/n^2}C(\mfm,\f)C(\mfm,\g)\right).$$
 	In the above summation $\mfb_m^{\mfc_1}(\f_1,\g_1)\geq 0$ for all $m$ because $C(\mfm,\f)C(\mfm,\g)\geq 0$, for all $(\mfm,\mfn)=1$,
 	by~\eqref{positive}. Observe that $\mfb_1^{\mfc_1}(\f_1,\g_1)=C(\mcO_F,\f)C(\mcO_F,\g)$.
 	
 	Denote $k_0 := \mathrm{max}\{k_1, k_2, \ldots, k_n\}$ and $l_0 := \mathrm{max} \{l_1, l_2, \ldots, l_n\}$.
 	Define, for any $j$, $k_j^{\prime} := k_0 - k_j$, and similarly, define $l_j^{\prime}$.
 	
 	Now, look at the complete $L$-function, defined by the product
 	$$\Lambda_{\f_1,\g_1}(s)= \prod_{j=1}^n \Gamma\left(s+1+ \frac{k_j-l_j-k_0-l_0}{2}\right)\Gamma\left(s-\frac{k^{\prime}_j+l^{\prime}_j}{2}\right)L_{\f_1,\g_1}(s)$$ 
 	can be continued to a holomorphic function on the whole plane, since the weights are different (cf.~\cite[Proposition 4.13]{Shi78}).
 	As the $\Gamma$-function is extended by analytic continuation to all complex numbers except the non-positive integers, where the function has simple poles,
 	we get that that function $L_{\f_1,\g_1}(s)$ is also entire. 
 	
 	By Landau's Theorem it follows that the Dirichlet series $L_{\f_1,\g_1}(s)$ converges everywhere. 
 	Observe that the function $L_{\f_1,\g_1}(s)$ has real zeros because the $\Gamma$-factors have poles at non-positive integers.
 	By Lemma~\ref{lem1}, we have that $\mfb_m^{\mfc_1}(\f_1,\g_1)=0$ for all $m$. 
 	This contradicts the assumption that $C(\mcO_F,\f)C(\mcO_F,\g) \neq0$ This completes the proof of $~\eqref{negative}$.

 	In order to complete the proof of the Theorem ~\ref{maintheorem1}, we need to show that there exist infinitely many $\mfm\subseteq\mcO_F$ such that 
 	$$\frac{C(\mfm,\f)C(\mfm,\g)}{C(\mcO_F,\f)C(\mcO_F,\g)}> 0.$$
 	It is sufficient to assume that $C(\mcO_F,\f)C(\mcO_F,\g)>0$. We then have to show that there exist infinitely many $\mfm$
 	such that $C(\mfm,\f)C(\mfm,\g)>0$. If not, then $C(\mfm,\f)C(\mfm,\g)\leq 0$ for all ideals $\mfm \subseteq \mcO_F$ with $N(\mfm) \gg 0$.
 	Note that, $C(\mfm,\f)C(\mfm,\g)$ cannot be equal to zero for almost all ideals $\mfm\subseteq\mcO_F$. 
 	For in this case $\sum_{\mfm\subseteq\mcO_F}\frac{C(\mfm,\f)C(\mfm,\g)}{N(\mfm)^s}$ is a Dirichlet polynomial and 
 	$$\Lambda_{\f,\g}(s)= \prod_{j=1}^n \Gamma(s+1+ \frac{k_j-l_j-k_0-l_0}{2})\Gamma(s-\frac{k^{\prime}_j+l^{\prime}_j}{2})L_{\f,\g}(s)$$ 
 	is entire. 
 	The presence of the multiple $\Gamma$-factors ensures that $\sum_{\mfm\subseteq\mcO_F}\frac{C(\mfm,\f)C(\mfm,\g)}{N(\mfm)^s}$ has infinitely many zeros.
 	Hence, by Lemma \ref{lem2},  we get that $C(\mcO_F,\f)C(\mcO_F,\g)=0$, which is a contradiction. Hence, there exists an integral ideal $\mfd \subseteq \mcO_F$ such that $C(\mfd,f)C(\mfd,g)<0$.  
 	
 	
 	Now, by Proposition ~\ref{Uoperator}, $\f|U(\mfd)$ and $\g|U(\mfd)$ are Hilbert cusp forms of weights $k_1, k_2$, respectively and weight $\mfd\mfc$.
 	Observe that $$C(\mcO_F,\f|U(\mfd))C(\mcO_F,\g|U(\mfd))=C(\mfd,\f)C(\mfd,\g)<0.$$
 	Now, by~\eqref{negative}, we have $C(\mfm,\f|U(\mfd))C(\mfm,\g|U(\mfd))>0$ for infinitely many $\mfm \subseteq\mcO_F$. This proves our claim.

 \end{proof}

\section{Proof of Theorem~\ref{maintheorem2}}
In this section, we shall prove Theorem~\ref{maintheorem2}.
 By~\cite[(2.23)]{Shi78}, the Fourier coefficients $C(\mfm,\f)$ of $\f$ satisfy the following Hecke relations
 \begin{equation*}
    C(\mfm,\f)C(\mfn,\f)=\sum_{\mfm+\mfn\subset \mfa}N(\mfa)^{k_0-1}C(\mfa^{-2}\mfm\mfn),
 \end{equation*} 
 where $k_0 = \mathrm{max}\{k_1, \ldots, k_n\}$. In particular, for any $m \geq 1$, the following relation holds:
 \begin{equation}
    \label{multiplication}
 C(\mfp^{m+1},\f)=C(\mfp,\f)C(\mfp^m,\f)-N(\mfp)^{k_0-1}C(\mfp^{m-1},\f).
 \end{equation}
 For any integral ideal $\mfa \subseteq \mcO_F$, define
 $$\beta(\mfa,f):=\frac{C(\mfa,\f)}{N(\mfa)^{\frac{k_0-1}{2}}}.$$ 
%
For any prime ideal  $\mfp\subseteq\mcO_F$, by~\eqref{multiplication}, we have the following 
\begin{equation}
  \label{relation} 
  \beta(\mfp^{m+1},\f)=\beta(\mfp,\f)\beta(\mfp^m,\f)-\beta(\mfp^{m-1},\f).
\end{equation}

It is well-known that for a primitive Hilbert cuspidal eigenform $\f$ over $F$, there is an irreducible cuspidal automorphic representation 
$\Pi=\Pi_{\f}$ of $GL_2(\A_F)$ corresponding to it.
For any place $\mfp$ of $F$ such that $\Pi_{\mfp}$ is unramified, let $\lambda_\mfp(\f)$ denote the eigenvalue of the Hecke operator
\begin{equation}\label{Heckeoperator}
 \GL_2(\mcO_\mfp)\left(\begin{matrix} \varpi_\mfp & \\ & 1 \end{matrix}\right) \GL_2(\mcO_\mfp)
\end{equation}
on $\Pi_{\mfp}^{\GL_2(\mcO_{\mfp})}$, where $\varpi_\mfp$ is a uniformizer of $\mcO_\mfp$. For such a prime $\mfp$,
by~\cite[Lemma 3.2]{KKT18}, we see that the eigenvalues $\lambda_\mfp(\f)$  and the Fourier coefficient $C(\mfp,\f)$ are related by 
 $$\lambda_\mfp(\f)=\frac{C(\mfp,\f)}{N(\mfp)^{\frac{k_0-2}{2}}}.$$
 For any fixed prime ideal $\mfp \nmid \mfc_1\mfc_2\Dif_F$, by~\cite[Theorem 3.3]{KKT18}, we have
 \begin{equation}\label{compare}
   \beta(\mfp,\f) :=\frac{\lambda_\mfp(\f)}{N(\mfp)^\frac{1}{2}}=\frac{C(\mfp,\f)}{N(\mfp)^{\frac{k_0-1}{2}}}\in[-2,2].
 \end{equation} 
Since $\beta(\mfp,\f) \in [-2,2]$, we can write $\beta(\mfp,\f) = 2 \cos \alpha_{\mfp}$,
for some $0 \leq \alpha_{\mfp} \leq \pi$. Before getting into the proof of Theorem~\ref{maintheorem2}, we need the following proposition.

\begin{prop}
\label{for:fg}
For any fixed prime ideal $\mfp \nmid \mfc_1\mfc_2\Dif_F$ and for any $m \geq 1$, we have
   \begin{equation}
     \beta(\mfp^\mfm,\f)=
   	 \begin{cases}
   	    (-1)^m(m+1) & {\rm if} \ \alpha_\mfp=\pi; \\
   	    m+1 & {\rm if} \ \alpha_\mfp=0; \\
   	    \frac{\sin(m+1)\alpha_\mfp}{\sin\alpha_\mfp} & {\rm if} \ 0<\alpha_\mfp<\pi.
   	 \end{cases}
   \end{equation} 
\end{prop}
\begin{proof}
The first two cases are easy to prove by induction. 
So WLOG assume that $0<\alpha_\mfp<\pi$. 
When $m=1$, we have $\beta(\mfp,\f)=\frac{\sin 2\alpha_\mfp}{\sin \alpha_\mfp}=\frac{2\sin \alpha_\mfp\cos \alpha_\mfp}{\sin \alpha_\mfp}=2 \cos \alpha_\mfp$.
Assume that $\beta(\mfp^\mfm,\f)=\frac{\sin(m+1)\alpha_\mfp}{\sin \alpha_\mfp}$ for some $m \geq 1$. By~\eqref{relation}, we have 
\begin{align*}
	\beta(\mfp^{m+1},\f) &= \beta(\mfp,\f) \beta(\mfp^m,\f)-\beta(\mfp^{m-1},\f)\\ 
	                       &=2 \cos\alpha_{\mfp}\frac{\sin(m+1)\alpha_\mfp}{\sin \alpha_\mfp}-\frac{\sin m\alpha_\mfp}{\sin \alpha_\mfp}\\ 
	                       &=\frac{2\sin(m+1)\alpha_\mfp\cos\alpha_\mfp-\sin m\alpha_\mfp}{\sin\alpha_\mfp}\\ 
	                       &=\frac{\sin(m+2)\alpha_\mfp+\sin m\alpha_\mfp-\sin m\alpha_\mfp}{\sin\alpha_\mfp}\\
	                       &=\frac{\sin(m+2)\alpha_\mfp}{\sin\alpha_\mfp}.
\end{align*}    
\end{proof}

Now, we are in a position to prove Theorem~\ref{maintheorem2}. Let $\mfp \nmid \mfc_1\mfc_2\Dif_F$ be a prime ideal.
   By~\eqref{compare}, one can write 
       $$\beta(\mfp,\f)=2 \ {\rm cos} \ \alpha_\mfp \ \rm{and} \ \beta(\mfp,\g)=2 \ {\rm cos} \ \beta_\mfp$$
   with $0\leq \alpha_\mfp,\beta_\mfp\leq \pi $.
   Now, the proof of Theorem~\ref{maintheorem2} follows from following cases. 
   \par

 \textbf{Case(1):} When $\alpha_\mfp=0$ or $\pi$ and $\beta_\mfp=0$ or $\pi$, then by Proposition~\ref{for:fg}, we see that 
     $$\{m\in\N | C(\mfp^m,\f)C(\mfp^m,\g)\neq 0 \}=\N.$$
 In this case all elements of the sequence $\{ C(\mfp^m,\f)C(\mfp^m,\g) \}_{m\in \N}$ are non-zero.
 
 \textbf{Case (2):} Suppose that at least one of $\alpha_\mfp$, $\beta_\mfp$ is $0$ or $\pi$, say $\alpha_\mfp=0$ or $\pi$ and $\beta_\mfp\in(0,\pi)$. If $\beta_\mfp/\pi\notin \Q$, there is nothing to prove. If $\beta_\mfp/\pi\in\Q$, say $\beta_\mfp=\frac{r}{s}$, where $r,s\in \N$ and $(r,s)= 1$, then we have ${\rm sin} \ m\alpha_\mfp=0$ if and only if $m$ is an integer multiple of $s$, then we have 
  $$\#\{m\leq x | C(\mfp^m,\f)C(\mfp^m,\g)\neq 0 \}=\#\{m\leq x |C(\mfp^m,\g)\neq 0 \}=[x]-\left[\frac{x}{s}\right].$$
  Hence the set $\{m\in \N | C(\mfp^m,\f)C(\mfp^m,\g)\neq 0\}$ has positive density. 

 \textbf{Case (3):} Suppose that $\alpha_\mfp=\beta_\mfp\in(0,\pi)$, i.e., $\alpha_\mfp/\pi=\beta_\mfp/\pi\in(0,1).$ 
  If $\alpha_\mfp/\pi\notin\Q$, then $C(\mfp^m,\f)C(\mfp^m,\g)\neq 0$ for all $m\in \N$ as ${\rm sin} \ m\alpha_\mfp\neq 0$ for all $m\in\N$.
  If $\alpha_\mfp/\pi\notin\Q$, say $\alpha_\mfp=\frac{r}{s}$, where $r,s\in \N$ and $(r,s)= 1$, then we have ${\rm sin} \ m\alpha_\mfp=0$ if and only if $m$ is an integer multiple of $s$ and hence 
 $$\#\{m\leq x | C(\mfp^m,\f)C(\mfp^m,\g)\neq 0 \}=[x]-\left[\frac{x}{s}\right].$$
 Hence the set $\{m\in \N | C(\mfp^m,\f)C(\mfp^m,\g)\neq 0\}$ has positive density.
 
 \textbf{Case (4):} Suppose that $\alpha_\mfp,\beta_\mfp\in(0,\pi)$ with $\alpha_\mfp\neq\beta_\mfp$. If both $\alpha_\mfp/\pi,\beta_\mfp/\pi\notin\Q$, then there is nothing to prove.
  Next suppose that one of them, say $\alpha_\mfp/\pi=\frac{r}{s}$ with $(r,s)=1$ and $\beta_\mfp\notin\Q$. Then we have 
 $$\#\{m\leq x | C(\mfp^m,\f)C(\mfp^m,\g)\neq 0 \}=\#\{m\leq x |C(\mfp^m,\f)\neq 0 \}=[x]-\left[\frac{x}{s}\right].$$
 Hence the set $\{m\in \N | C(\mfp^m,\f)C(\mfp^m,\g)\neq 0\}$  has positive density. 
 
 Now let both $\alpha_\mfp/\pi,\beta_\mfp/\pi\in\Q$. If $\alpha_\mfp/\pi=\frac{r_1}{s_1}$ and $\beta_\mfp/\pi=\frac{r_2}{s_2}$ with $(r_i,s_i)=1$, for $1\leq i\leq 2$, then 
 $$\#\{m\leq x | C(\mfp^m,\f)C(\mfp^m,\g)\neq 0 \}=\#[\{m\leq x |C(\mfp^m,\f)\neq 0 \}\cap \{m\leq x |C(\mfp^m,\g)\neq 0 \}].$$
 Since 
 \begin{align*}
 \#\{m\leq x | C(\mfp^m,\f)C(\mfp^m,\g)= 0 \}&=\#[\{m\leq x |C(\mfp^m,\f)= 0 \}\cup \{m\leq x |C(\mfp^m,\g)= 0 \}]\\
  &\leq \left[\frac{x}{s_1}\right]+\left[\frac{x}{s_2}\right].
 \end{align*}
 Hence the set $\{m\in \N | C(\mfp^m,\f)C(\mfp^m,\g)\neq 0\}$  has positive density. 
 This completes the proof of Theorem~\ref{maintheorem2}.

\end{document}